\documentclass[a4paper,10pt]{article}
\usepackage{cmap}
\usepackage[T2A]{fontenc}
\usepackage[utf8]{inputenc} 
\usepackage[english]{babel}
\usepackage{amssymb,amsmath}
\usepackage{graphicx}
\usepackage{color}
\usepackage{ifthen}
\usepackage{multirow}
\makeatletter
\def\@seccntformat#1{\csname the#1\endcsname.\ } 
\def\@biblabel#1{#1.} 
\makeatother

\date{}

\newenvironment{proof}[1][\hspace{-1.0ex}]%
{\par\addvspace{1mm}{\sc Proof\hspace{1.0ex}{#1}.} }%
{\quad$\blacktriangle$\par\addvspace{1mm}}

\newif\ifNoRemark
\def\addtheorem#1#2#3#4{
\ifthenelse{\equal{#2}{}}{}%
{\ifthenelse{\expandafter\isundefined\csname the#2\endcsname}{\newcounter{#2}}{}}
\newenvironment{#1}[1][\global\NoRemarktrue]
{\par\addvspace{2mm plus 0.5mm minus 0.2mm}\noindent 
{\bf #3}\ifthenelse{\equal{#2}{}}{}%
{\refstepcounter{#2}{\bf ~\csname the#2\endcsname}}%
{\bf \vphantom{##1}\ifNoRemark.\ \else\ (##1).\fi}\begingroup #4}%
{\endgroup\par\addvspace{1mm plus 0.5mm minus 0.2mm}\global\NoRemarkfalse}
\expandafter\newcommand\csname b#1\endcsname{\begin{#1}}
\expandafter\newcommand\csname e#1\endcsname{\end{#1}}
}

\addtheorem{theorem}{thrm}{Theorem}{\sl}
\addtheorem{lemma}{lmm}{Lemma}{\sl}
\addtheorem{pro}{prpstn}{Proposition}{\sl}
\addtheorem{proposition}{prpstn}{Proposition}{\sl}
\addtheorem{problem}{problem}{Problem}{\sl}
\addtheorem{example}{exmpl}{Example}{\sl}

\title{On completely regular codes of covering radius $1$\\
in the halved hypercubes%
\thanks{This work was funded by the Russian Science Foundation under grant 18-11-00136.}}
\author{Denis S. Krotov, Ivan Yu. Mogilnykh, Anastasia Yu. Vasil'eva
\thanks{Sonolev Institute of Mathematics, pr. Akademika Koptyuga 4, Novosibirsk 630090, Russia.}
}

\newcommand\Hf[1]{\frac12Q_{#1}}
\newcommand\Hm[1]{Q_{#1}}
\newcommand{\Id}{\mathrm{Id}}
\newcommand{\Mx}[4]{\left[ \begin{array}{cc} #1&#2 \\ #3&#4 \end{array}\right]}
\newcommand{\Mxx}[4]{\left[ \begin{array}{@{}c@{\,}c@{}} #1&#2 \\ #3&#4 \end{array}\right]}
\newcommand{\Mxxx}[4]{\left[ \begin{array}{@{\,}c@{\ }c@{\,}} #1&#2 \\ #3&#4 \end{array}\right]}
\begin{document}

\maketitle
\begin{abstract}
We consider constructions of covering-radius-$1$ completely regular codes, or, equivalently,
equitable $2$-partitions, of halved $n$-cubes.

\end{abstract}

\section{Definitions.}
An equitable $k$-partition  of the graph is an ordered partition
$C=(C_0,C_1,\ldots,C_{k-1})$ of the vertex set such that the
number of neighbors in $C_j$ of a vertex from $C_i$ is some
constant $S_{ij}$, $i,j=0,\ldots,k-1$; the matrix
$S=\{S_{ij}\}_{i,j\in \{0,\ldots,k-1\}}$ is called the
\emph{quotient matrix} (sometimes, the intersection matrix, or the
parameter matrix). The elements of
 the partition $(C_0,C_1,\ldots,C_{k-1})$ are called \emph{cells}.

We will write $C(\bar v)=i$ if $\bar v\in C_i$ and identify a
$k$-partition with the corresponding $k$-valued function on the
vertices.

A set $C$ (code) of vertices of a graph is called \emph{completely regular}
if the partition $(C^{(0)},C^{(1)},\ldots,C^{(\rho)})$ is equitable,
where $C^{(d)}$ is the set of vertices at distance $d$ from $C$ and $\rho$
(the \emph{covering radius} of $C$) is the maximum $d$ for which $C^{(d)}$
is nonempty.

A connected graph is called distance regular if every singleton
(one-vertex set) is completely regular with the same quotient matrix.

The \emph{$n$-cube} (or $n$-dimensional \emph{hypercube}) $\Hm{n}$
is the graph on the binary words of length $n$
(also treated as vectors over the binary field GF$(2)$)
where two words are adjucent iff they differ in exactly one position.

The \emph{halved $n$-cube} $\Hf{n}$ (of $\Hf{n}'$) is the graph
on the binary words of length $n$ with even number of ones
(or odd for $\Hf{n}'$;
the resulting graph is the same up to isomorphism)
where two words $\bar x$, $\bar y$ are adjacent (we write $\bar x\sim\bar y$) iff they differ in exactly two positions.

An $s$-face of the halved $n$-cube is the set of $2^{s-1}$
vertices that have the same values in some $n-s$ ``fixed''
coordinates. The remaining coordinates are called \emph{free} for
the $s$-face.

The eigenvalues of the halved $n$-cube are denoted
$\theta_0(n)$, \ldots, $\theta_{\lfloor n/2 \rfloor}(n)$,
in the decreasing order, $\theta_i(n)=((n-2i)^2-n)/2$.

In a graph, a set of mutually adjacent vertices is called a \emph{clique}.
In a distance-regular graph with the minimum eigenvalue $\theta_{\min}$,
a clique of size $1-k/\theta_{\min}$ is called a \emph{Delsatre clique}.
In $\Hf{n}$ with even $n$, any Delsarte clique have size $n$.
For odd $n$, $\Hf{n}$ has no Delsarte cliques.
\section{Necessary conditions.}\label{s:Nec}\label{s:necessary}
In this section, we collect the necessary conditions for a
two-by-two matrix to be a quotient matrix $S=\Mxx{a}{b}{c}{d}$ of
an equitable $2$-partition $C=(C_0,C_1)$.

\begin{enumerate}
 \item Firstly, all elements are nonnegative integers, $b$ and $c$ are positive,
$a+b$ and $c+d$ equals the degree of the graph (for $\Hf{n}$, the
degree is $n(n-1)/2$).
\item Next, by counting the edges between $C_0$ and $C_1$, we find $|C_0|/|C_1|=b/c$.
It follows that $\frac{b+c}{\mathrm{gcd}(b,c)}$ divides the number
of vertices of the graph. For $\Hf{n}$, the number of vertices is
$2^{n-1}$; so, $\frac{b+c}{\mathrm{gcd}(b,c)}$ is a power of $2$.
\item Next, the eigenvalues of the quotient matrix are eigenvalues of the graph.
For $\Hf{n}$, this means that $a-c=d-b=((n-2i)^2-n)/2$, $i\in
\{0,\ldots,\lfloor n/2 \rfloor\}$.
\item Given a completely regular code $D$,
if we know the numbers $|C_i\cap D|$ for all $i\{0,\ldots,k-1\}$,
then we can calculate $|C_i\cap D^{(d)}|$ for any $i\in
\{0,\ldots,k-1\}$ and
 $d\in\{0,\ldots,\rho\}$ (see e.g. \cite{AvgMog:J63J73}) via the quotient
matrices of $C$ and $D$. The collection $\{|C_i\cap D^{(d)}|: i\in
\{0,\ldots,k-1\},
 d\in\{0,\ldots,\rho\}\}$ is called
the \emph{weight distribution} of the partition $C$ with respect
to $D$. An equitable partition $C$ necessarily possesses at least
one weight distribution with nonnegative integer elements. More on
calculation of weight distributions of equitable partitions in
distance regular graphs could be found in \cite{Kro:struct}.

The most important case is a distance regular graph when $D$ is a
singleton. For an equitable partition with the quotient matrix $S$
in $\Hf{n}$, we can conclude that all the matrices $S^{(0)}=\Id$,
$S^{(2)}=S$, $S^{(4)}$, \ldots, $S^{(2\lfloor n/2\rfloor)}$ are
related by the recursive identity
\begin{equation}\label{eq:2d}
 S\cdot S^{(i)} = {{n-i+2}\choose{2}} S^{(i-2)} + i(n-i)S^{(i)} +
{{i+2}\choose{2}} S^{(i+2)}
\end{equation}
(this formula is a partial case of \cite[(4)]{Kro:struct}) and
must consist of nonnegative integers.

\begin{example}\label{ex:1}
Consider $n=12$, $S=\Mx{4}{62}{2}{64}$. We find
$S^{(4)}=S^{(8)}=\Mx{-1}{496}{16}{479}$. Hence, there is no
equitable partition of $\Hf{12}$ with the quotient matrix $S$.
\end{example}

%
\end{enumerate}

\section{Special cases.}\label{s:spec}
\subsection{Odd $n$.}\label{ss:odd}
\begin{theorem}\label{th:odd}
 Let $C=(C_0,C_1,\ldots,C_{k-1})$ be an equitable $k$-partition of $\Hf{n}$, where $n$ is odd.
 Let $C_i'$ be $C_i\cup(C_i+\bar 1).$ Then $C'=(C'_0,\ldots,C'_{k-1}\}$ is an equitable $k$-partition of $\Hm{n}$.
\end{theorem}
\begin{proof}
 Consider a vertex $\bar x \in C_i+\bar 1$ for some $i$ (the case $\bar x \in C_i$ is similar).
 Its neighbors are exactly the vertices at distance $n-1$ from
 $\bar x +\bar 1$ in the graph $\Hm{n}$; at the same time, the vertices at distance $(n-1)/2$ from
 $\bar x +\bar 1\in C_i$ in the graph $\Hf{n}$.
 Since $C$ is an equitable partition and the graph $\Hf{n}$ is distance regular, the number of such vertices in $C_j$ for each $j$ is a constant
regardless of choice of $\bar x$ in $C_i+\bar 1$,
see e.g. \cite{Martin:PhD,Kro:struct}.
 By the definition,
we have an equitable partition of $\Hm{n}$.
\end{proof}
So, the existence of equitable partitions of the halved $n$-cube,
where $n$ is odd $n$,
with given putative quotient matrix $S$
is equivalent to the existence of equitable partitions of the $n$-cube
with some quotient matrix, uniquely defined from $S$ (for the case of $2$-partitions, this matrix can be easily found from the eigenvalue and the proportion between the partition cells).
We see that in this case, odd $n$, there is no sense to consider the problem of existence separately for the halved $n$-cube.
\subsection{Even $n$, minimum eigenvalue.}\label{ss:even-min}
\begin{theorem}\label{th:min}
 A partition $(C_0,C_1)$ of $\Hf{n}$, $n$ even, is equitable with the quotient matrix $S$ and the eigenvalue $\theta_{n/2}(n)$
 if and only if $(C'_0,C'_1)$ is an equitable partition of $\Hm{n-1}$ with the
 quotient matrix $S'=\Mx{c-1}{n-c}{c}{n-c-1}$ for some $c$ and $S=(S'^2-\Id)/2+S'$.
\end{theorem}
\begin{proof}
 Consider the graph $\Hm{n-1}$ and the graph $\Hm{n-1}^+$ with the same vertex set, where two vertices are adjacent iff the distance between them is $1$ or $2$ in  $\Hm{n-1}$. The adjacency matrices $A$ and $A^+$ of $\Hm{n-1}$ and $\Hm{n-1}^+$  are related
 as follows
 $$A^+=\frac{A^2+(n-1)\Id}2 + A = \frac{(A+\Id)^2}2+\frac{(n-2)}2\Id.$$
 It is straightforward that an eigenvector of $A^+$ corresponding to the minimum eigenvalue is an  eigenvector of $A$ corresponding to the  eigenvalue $-1$ and vice versa. It follows that the equitable $2$-partitions of   $\Hm{n-1}$ and $\Hm{n-1}^+$ with the eigenvalue $-1$ and with the minimum eigenvalue, respectively, are also the same.

 It remains to note that $\Hf{n}$  is isomorphic to $\Hm{n-1}^+$
 (an isomorphism is given by removing the last symbol from each vertex word).
\end{proof}

So, equitable $2$-partitions of $\Hf{n}$, $n$ even, with the minimum eigenvalue are in one-to-one correspondence with the equitable $2$-partitions of $\Hm{n-1}$  with the  eigenvalue $-1$, which are constructed in \cite{FDF:12cube.en} for all admissible parameters.


\section{Constructions.}\label{s:constr}

\subsection{From equitable partitions of $\Hm{n}$.}\label{ss:Hm}\label{s:Hm}

It is known that an equitable partition of an $n$-cube
restricted to the vertices of $\Hf{n}$ is also equitable.

In particular, any equitable $2$-partition of $\Hm{n}$ different from the partition into bipartite parts induces an equitable $2$-partition of $\Hf{n}$. The parameters of equitable $2$-partition of $\Hm{n}$ were studied in
\cite{FDF:PerfCol},
\cite{FDF:12cube.en},
\cite{FDF:CorrImmBound}.

Another important case is the even-weight completely regular codes
of covering radius $3$ and $4$ in  $\Hm{n}$. In case of radius
$3$, the code $C$ and the set $C^{(2)}$  form an equitable
partition of $\Hf{n}$. In the case of radius $4$, the sets
$C^{(1)}$ and $C^{(3)}$ form an equitable partition of $\Hf{n}'$.
Below are examples of completely regular codes of covering radius
$4$:
\begin{itemize}
 \item The repetition code $\{00000000, 11111111\}$ in $\Hm{8}$.
 \item The (extended) Preparata codes of minimal distance $6$ in $\Hm{4^m}$ \cite{BZZ:1974:UPC,SZZ:1971:UPC}.
 \item The extended BCH codes
 of minimum distance $6$ in $\Hm{2^{2m+1}}$ \cite{BZZ:1974:UPC,SZZ:1971:UPC}.
 \item The minimum-distance-$6$ Hadamard code in $\Hm{12}$ \cite{BasZin77}.
 \item The minimum-distance-$8$ Golay code in $\Hm{24}$.
\end{itemize}

\subsection{Linear partitions.}\label{s:linear}\label{ss:linear}
A $2$-partition is called linear if its first cells is a linear code.
It is straightforward to show that a linear code induces an equitable partition of  $\Hf{n}$ if and only if it has a parity check matrix of form $\left({\bar 1 \atop B} \right)$, where $\bar 1$ is the all-one row and the matrix $B$ is such that every non-zero column of the same weight is represented a constant number of times as the sum of two columns of $B$ (it is not required that the columns of $B$ are distinct).

\subsection{Union.}\label{ss:union}\label{s:union}
\begin{lemma}\label{l:union}
If $(C_0,C_1)$ and $(P_0,P_1)$ are equitable $2$-partitions of the
same regular graph with cospectral quotient matrices
$\Mx{a'}{b'}{c'}{d'}$, $\Mx{a''}{b''}{c''}{d''}$ and $C_0\cap
P_0=\emptyset$, then  $(C_0\cup P_0,C_1\cap P_1)$ is an equitable
$2$-partition with the quotient matrix $\Mx{a}{b}{c}{d}$,
$c=c'+c''$, $a=a'+c''=a''+c'$, $b=...$, $d=...$.
\end{lemma}
In particular,
\begin{lemma}\label{l:lunion}
 If $(C_0,C_1)$ is a linear equitable $2$-partition of $\Hf{n}$
with the quotient matrix $\Mx{a}{b}{c}{d}$, then we can unify any
number $t< 2^{n-1}/|C_0|$ of cosets of $C_0$ to get an equitable
partition with the quotient matrix
$\Mx{a+(t-1)c}{b-(t-1)c}{tc}{d-(t-1)c}$.

\end{lemma}
Another important case is when we have several even-weight
completely regular codes $C_1$, $C_2$, \ldots, $C_l$ of covering
radius $4$ at distance $4$ from each other (we imply that the
quotient matrix is the same for all these codes). In this case,
$(C_1^{(1)}\cap C_2^{(1)}\cap\ldots \cap C_l^{(1)}, C_1^{(3)}\cup
C_2^{(3)}\cup\ldots \cup C_l^{(3)})$ is an equitable partition of
$\Hf{n}'$. Examples of such collections of completely regular
codes can be constructed based on cosets of distance-$6$ BCH codes
or Preparata codes (it is known that the extended Hamming code in
$\Hm{2^m}$ can be partitioned into cosets of BCH or Preparata
codes \cite{BvLW}, depending on the parity of $m$). Even more
interesting example is the Golay code $G$, when translations of
$G^{(1)}$ can be combined with cosets of some linear code,
providing a large family of quotient matrices \cite{Kro:24}.

\subsection{$\times t$ construction.}\label{s:*t}
This simple construction works for many Cayley graphs. For
equitable partitions of hypercube, it was implemented
in~\cite{FDF:PerfCol}. In the case of the halved hypercube, the
construction is the same; the difference is only in calculating
parameters.

\begin{lemma}[\cite{FDF:PerfCol}]\label{l:*t0}
 Let $C=(C_0,C_1,\ldots,C_{k-1})$ be an equitable partition of $\Hm{n}$
 with the quotient matrix $S$.
Then the partition $C^{(t)}$ is an equitable partition of
$\Hm{tn}$
  with the quotient matrix $tS$, where
  $$ C^{(t)}(\bar x_1,\ldots, \bar x_t) := C(\bar x_1 + \ldots + \bar x_t).$$
 Further,
 if $\theta_i(n)$ is an eigenvalue of $S$,
 then $\theta_{ti}(tn)$ is an eigenvalue of $tS$.
\end{lemma}

\begin{lemma}\label{l:*t}
 Let $C=(C_0,C_1,\ldots,C_{k-1})$ be an equitable partition of $\Hf{n}$
 with the quotient matrix $S$.
 Then $C^{(t)}$ is an equitable partition of $\Hf{tn}$
 with the quotient matrix $S^{(t)}$, where
 $$ C^{(t)}(\bar x_1,\ldots, \bar x_t) := C(\bar x_1 + \ldots + \bar x_t),
 \qquad    S^{(t)} = t^2S+ n\frac{t(t-1)}{2} \Id . $$
\end{lemma}
\begin{proof}
The claim follows from the definitions
and the following direct facts.
\begin{itemize}
 \item
 Each vertex $\bar v$ of $\Hf{n}$ corresponds to the set
 $$t(\bar v):= \{(\bar x_1,\ldots, \bar x_t) \mid
 \bar x_1 + \ldots + \bar x_t = \bar v \} $$
 of vertices of $\Hf{tn}$;
 moreover, for all $\bar x$ in $t(\bar v)$ one has $f^{(t)}(\bar x)=f(\bar v)$.
 \item
 If $\bar v \sim \bar v'$,
 then every vertex from $t(\bar v)$ has exactly $t^2$ neighbors in $t(\bar v')$.
 \item
 If $\bar v \not\sim \bar v'$,
 then every vertex from $t(\bar v)$ has no neighbors in $t(\bar v')$.
 \item
 Every vertex from $t(\bar v)$ has exactly $n\frac{t(t-1)}2$ neighbors
 in $t(\bar v)$.
\end{itemize}
The calculation of eigenvalues is straightforward.
\end{proof}

\subsection{Splitting construction.}\label{s:FDF}\label{ss:FDF}
The construction in this section can be considered as a variant of the main construction in~\cite{FDF:12cube.en}.
However, in the case of the halved hypercube,
it gives an equitable $2$-partition only if the parameters are connected by some equation.
Assume that we have an equitable $2$-partition $(P_0,P_1)$ of
$\Hf{n}$ with the quotient matrix $\Mx{a}{b}{c}{d}$. Moreover,
assume that for some $s$, the cell $P_1$ can be partitioned into
$s$-faces, let $F(\bar v)$ denote the face from the partition such
that $\bar v\in F(\bar v)$.

\def\sgn{\mathrm{sign}}
\begin{lemma}\label{l:split1}
Define
 $$ C^{(2)}(\bar x, \bar y) := C(\bar x+ \bar y)$$
 and for $\bar x \in P_1$, define
 $$ \sgn(\bar x, \bar y) := x_1+\ldots+x_n+b_1y_1 +\ldots + b_ny_n ,$$
 where $\bar b=(b_1,...,b_n)$ is defined by the directions of the
 $s$-face $F=F(\bar x+ \bar y)$: it has $1$s
 in the free coordinates of $F$ and $0$s in the fixed coordinates.
 Then
 $$C'(\bar v) :=
 \begin{cases}
  0 &  \mbox{ if $ f^{(2)}(\bar v)=0$} \\
  1 &  \mbox{ if $ f^{(2)}(\bar v)=1$ and $ \sgn(\bar v)=0$}\\
  2 &  \mbox{ if $ f^{(2)}(\bar v)=1$ and $ \sgn(\bar v)=1$}
 \end{cases}
 $$
 is an equitable partition of $\Hf{tn}$
 with the quotient matrix
$$ S' =
\left[
 \begin{array}{ccc}
  4a+n & 2b   & 2b \\
  4c   & 2d   + s^2   & 2d+n - s^2 \\
  4c   & 2d+n - s^2   & 2d   + s^2
 \end{array}\right]
$$
\end{lemma}
\begin{proof}
The claim follows from the definitions
and the following direct facts.
\begin{itemize}
 \item
 Each vertex $\bar v$ of $\Hf{n}$ corresponds to the set
 $$t(\bar v):= \{(\bar x, \bar y) \mid
 \bar x + \bar y = \bar v \} $$
 of vertices of $\Hf{2n}$;
 moreover, for all $\bar x$ in $t(\bar v)$ one has $f^{(t)}(\bar x)=f(\bar v)$.
 \item
 If $\bar v \sim \bar v'$,
 then every vertex from $t(\bar v)$ has exactly $4$ neighbors in $t(\bar v')$.
 \begin{itemize}
  \item[$\bullet\bullet$]
  Moreover, if $f(\bar v')=1$ and $\bar v\not\in F(v')$, then for two of these
  neighbors $\sgn$ equals $0$ and for two other equals $1$.
  \item[$\bullet\bullet$]
  Moreover, if $f(\bar v')=1$ and $\bar v\in F(v')$,
  then for all $4$ these neighbors the value of $\sgn$ is the same
  as for the vertex itself.
  \end{itemize}
 \item
 If $\bar v \not\sim \bar v'$,
 then every vertex from $t(\bar v)$ has no neighbors in $t(\bar v')$.
 \item
 Every vertex from $t(\bar v)$ has exactly
 $n$ neighbors in $t(\bar v)$.
 \begin{itemize}
 \item[$\bullet\bullet$]
 Moreover, if  $f(\bar v)=1$, then $\sgn$
 of these neighbors have the same value
 of $s$ and the other $n-s$, the different value.
 \end{itemize}
\end{itemize}
\end{proof}

If $2b=2d+n-s^2$, then we can merge the first two cells of the
obtained $3$-partition and get an equitable $2$-partition:
\begin{theorem}\label{th:fdf}
Let $(P_0,P_1)$ be an equitable $2$-partition of $\Hf{n}$ with the
quotient matrix $\Mx{a}{b}{c}{d}$ and $P_1$ be partitioned into
$s$-faces, where $s=\sqrt{2d+n-2b}$. Then there exists an
equitable $2$-partition of $\Hf{2n}$ with the quotient matrix
$\Mx{4a+n+2b}{2b}{4c+2b}{4d+n-2b}$.
\end{theorem}

For example, take $n=6$, $S_c=\Mx{c-1}{16-c}{c}{15-c}$, $s=2$,
$c=0,2,...,14$. In order to obtain an equitable partition of
$\Hf{12}$ with the quotient matrix
$\Mx{34+2c}{32-2c}{32+2c}{34-2c}$ we construct an equitable
partition $(P_0,P_1)$ of $\Hf{6}$ with the quotient matrix $S_c$
such that $P_1$ is partitioned into $2$-faces (which are just
pairs of adjacent vertices). One can see that $C=\{\bar 0, \bar
1\}$ and its complement gives $S_1$. So, the union of $c$ cosets
of $C$ gives the matrix $S_c$. If $c$ is even, then we can take
each of $P_0$ and $P_1$ as the union of ($c/2$ and $8-c/2$,
respectively) cosets of $\{000000,111111,000011,111100\}$, which
is obviously partitioned into edges. If $c$ is odd, we take $P_1$
as the union of
$$\{000011,111100,001100,110011,110000,001111\}=
\{000011,001111 \}\cup \{001100,111100\}\cup \{110000,110011\}$$
and $(13-c)/2$ cosets of $\{000000,111111,000011,111100\}$. We conclude that equitable partitions of $\Hf{12}$
with the following quotient matrices exist:
$$
\Mxx{34}{32}{32}{34}
\Mxx{36}{30}{34}{32}
\Mxx{38}{28}{36}{30}
\Mxx{40}{26}{38}{28}
\Mxx{42}{24}{40}{26}
\Mxx{44}{22}{42}{24}
\Mxx{46}{20}{44}{22}
\Mxx{48}{18}{46}{20}
\Mxx{50}{16}{48}{18}
\Mxx{52}{14}{50}{16}
\Mxx{54}{12}{52}{14}
\Mxx{56}{10}{54}{12}
\Mxx{58}{8}{56}{10}
\Mxx{60}{6}{58}{8}
\Mxx{62}{4}{60}{6}
$$

\section{Computational results.}

Let $C=(C_0,C_1,\ldots,C_{k-1})$ be a partition of the vertex set
a graph $G$. Let $\chi_{C_{i}}$ denote the characteristic vector
of the subset $C_i$ of the vertex set of $G$. It is easy to see
that $C$ is equitable with the quotient matrix $S$ iff

\begin{equation}\label{eq:eqp} A\left[
 \begin{array}{cccc}
  \chi_{C_0} & \chi_{C_1}&\ldots   & \chi_{C_{k-1}} \\
 \end{array}\right]=\left[
 \begin{array}{cccc}
  \chi_{C_0} & \chi_{C_1}&\ldots   & \chi_{C_{k-1}} \\
 \end{array}\right]S.
 \end{equation}
The existence problem of an equitable partition could be treated
as a binary linear programming problem with the variables
$\chi_{C_0}, \chi_{C_1},\ldots,\chi_{C_{k-1}}$ and the constrains
(\ref{eq:eqp}). The following fact was settled using GAMS
\cite{GAMS}.

\begin{theorem}\label{th:no10}
There are no equitable $2$-partitions of $\Hf{10}$ with quotient
matrix $\Mx{a}{b}{c}{d}$ and eigenvalue $13$ for any
$a\in\{14,\ldots,28\}\setminus 21$.
 \end{theorem}

\section{Quotient matrices for small even $n$}
We list all matrices $\Mx{a}{b}{c}{d}$ that satisfy $b\ge c$, conditions 1--3 of Section~\ref{s:necessary},
and, for the eigenvalue $\theta_{n/2}(n)$, 
the conclusion of Theorem~\ref{th:min}.

For the matrices in red and with index $-$, equitable partitions do not exist.
The index $+$ means the existence, $?$ means an open question.

Note that for any symmetric $2\times 2$ matrix with the proper eigenvalues,
an equitable partition exists in accordance with Sections~\ref{s:Hm} and~\ref{s:linear}.

\newcommand\Qxxx[4]{\Mxxx{\mathit{#1}}{\mathit{#2}}{\mathit{#3}}{\mathit{#4}}_?}
\newcommand\Exxx[4]{\Mxxx{{#1}}{{#2}}{{#3}}{{#4}}_+}

\subsection{$n=4$.} This case is left as an exercise.
\subsection{$n=6$.}
In $\Hf{6}$, all the equitable partitions with non-symmetric quotient matrices are constructed my merging cosets
of the linear perfect code $\{000000,111111\}$ (see also the end of Section~\ref{s:FDF}).

$\theta_1(6)=5$: $\Mx{10}{5}{5}{10}$.

$\theta_2(6)=-1$: 
$\Mx{7}{8}{8}{7}$, 
$\Mx{6}{9}{7}{8}$, 
$\Mx{5}{10}{6}{9}$, 
$\Mx{4}{11}{5}{10}$, 
$\Mx{3}{12}{4}{11}$, 
$\Mx{2}{13}{3}{12}$, 
$\Mx{1}{14}{2}{13}$, 
$\Mx{0}{15}{1}{14}$.

$\theta_3(6)=-3$: $\Mx{6}{9}{9}{6}$.

\subsection{$n=8$.}
For $\theta_2(8)$ in $\Hf{8}'$, partitions with all four matrices
can be constructed as $C^{(1)}$ for $C=\{00000000,11111111\}$, and the union
of $2$, $3$, or $4$ translations of $C^{(1)}$, according to Section~\ref{ss:union}.

$\theta_1(8)=14$: $\Mx{21}{7}{7}{21}$.

$\theta_2(8)=4$: 
$\Mx{16}{12}{12}{16}$, 
$\Mx{13}{15}{9}{19}$, 
$\Mx{10}{18}{6}{22}$, 
$\Mx{7}{21}{3}{25}$.

$\theta_3(8)=-2$: $\Mx{13}{15}{15}{13}$.

$\theta_4(8)=-4$: 
$\Mx{12}{16}{16}{12}$, 
$\Mx{8}{20}{12}{16}$, 
$\Mx{4}{24}{8}{20}$, 
$\Mx{0}{28}{4}{24}$.

\subsection{$n=10$.}
The existence results are from equitable $2$-partitions of $\Hm{10}$ and,
for $\theta_4(10)$, from a linear code and merging cosets of this code.

$\theta_1(10)=27$: 
$\Mx{36}{9}{9}{36}$.

$\theta_2(10)=13$: 
$\Mx{29}{16}{16}{29}$,
{\color{red}
$\color{red}\Mx{28}{17}{15}{30}_-$,
\ldots,
$\color{red}\Mx{22}{23}{9}{36}_-$,}
$\Mx{21}{24}{8}{37}$,
{\color{red}
$\color{red}\Mx{20}{25}{7}{38}_-$,
\ldots,
$\color{red}\Mx{14}{31}{1}{44}_-$}.

$\theta_3(10)=3$: 
$\Mx{24}{21}{21}{24}$.

$\theta_4(10)=-3$: 
$\Mxxx{21}{24}{24}{21}$,
$\Mxxx{18}{27}{21}{24}$,
$\Mxxx{15}{30}{18}{27}$,
$\Mxxx{12}{33}{15}{30}$,
$\Mxxx{ 9}{36}{12}{33}$,
$\Mxxx{ 6}{39}{ 9}{36}$,
$\Mxxx{ 3}{42}{ 6}{39}$,
$\Mxxx{ 0}{45}{ 3}{42}$.

$\theta_5(10)=-5$: 
$\Mx{20}{25}{25}{20}$.

\subsection{$n=12$.}
For $\theta_2(12)$, the only non-symmetric matrix with known existence comes from an equitable partition of $\Hm{12}$.
For $\theta_4(12)$, all matrices with even $b$ and $c$ were considered in Section~\ref{s:FDF}; for the nonexistence with $c=2$, see Example~\ref{ex:1}; for $c=1$, similarly.
The only known equitable partition $(H^{(1)},H^{(3)})$ 
of $\Hf{12}'$ with  $\theta_4(12)$ and odd $c$ and $b$
comes from the Hadamard code $H$ in $\Hm{12}$ 
as in Section~\ref{ss:Hm}.

$\theta_1(12)=44$:
$\Mxxx{55}{11}{11}{55}$.

$\theta_2(12)=26$:
$\Exxx{46}{20}{20}{46}$,
$\Qxxx{41}{25}{15}{51}$,
$\Exxx{36}{30}{10}{56}$,
$\Qxxx{31}{35}{5}{61}$.

$\theta_3(12)=12$:
$\Exxx{39}{27}{27}{39}$.

$\theta_4(12)=2$:
$\Exxx{34}{32}{32}{34}$,
$\Qxxx{33}{33}{31}{34}$,
$\Exxx{32}{34}{30}{34}$,
$\Qxxx{31}{35}{29}{34}$,
$\Exxx{30}{36}{28}{34}$,
$\Qxxx{29}{37}{27}{34}$,
$\Exxx{28}{38}{26}{34}$,
$\Qxxx{27}{39}{25}{34}$,
$\Exxx{26}{40}{24}{34}$,
$\Qxxx{25}{42}{23}{34}$,
$\Exxx{24}{42}{22}{34}$,
$\Qxxx{23}{42}{21}{34}$,
$\Exxx{22}{42}{20}{34}$,
$\Qxxx{21}{42}{19}{34}$,
$\Exxx{20}{42}{18}{34}$,
$\Qxxx{19}{42}{17}{34}$,
$\Exxx{18}{42}{16}{34}$,
$\Qxxx{17}{42}{15}{34}$,
$\Exxx{16}{52}{14}{34}$,
$\Qxxx{15}{52}{13}{34}$,
$\Exxx{14}{52}{12}{34}$,
$\Qxxx{13}{52}{11}{34}$,
$\Exxx{12}{52}{10}{34}$,
{\boldmath
$\Exxx{11}{52}{ 9}{34}$,
}
$\Exxx{10}{52}{ 8}{34}$,
$\Qxxx{ 9}{52}{ 7}{34}$,
$\Exxx{ 8}{52}{ 6}{34}$,
$\Qxxx{ 7}{52}{ 5}{34}$,
$\Exxx{ 6}{62}{ 4}{34}$,
$\Qxxx{ 5}{62}{ 3}{34}$,
{\color{red}
$\Mxxx{ 4}{62}{ 2}{34}_-$,
$\Mxxx{ 3}{62}{ 1}{34}_-$.
}

$\theta_5(12)=-4$:
$\Exxx{35}{31}{31}{35}$.

$\theta_6(12)=-6$:
$\Exxx{30}{36}{36}{30}$,
$\Exxx{12}{54}{18}{48}$.

\subsection*{Acknowledgments} 
The authors are grateful to Artem Panin for
sharing his experience of working with GAMS.

\nocite{Kro:24}
\nocite{FDF:PerfCol}
\nocite{FDF:12cube.en}
\nocite{FDF:CorrImmBound}


\providecommand\href[2]{#2} \providecommand\url[1]{\href{#1}{#1}}
  \def\DOI#1{{\small {DOI}:
  \href{http://dx.doi.org/#1}{#1}}}\def\DOIURL#1#2{{\small{DOI}:
  \href{http://dx.doi.org/#2}{#1}}}

\end{document}